\documentclass[11pt]{amsart}

\usepackage{amsfonts, amstext, amsmath, amsthm, amscd, amssymb}
\usepackage{graphicx, color, epsfig, subfigure, wrapfig, overpic}
\usepackage[all,cmtip]{xy} 
\usepackage{import}

\catcode `\@=11
\long\def\@savemarbox#1#2{\global\setbox#1\vtop{\hsize\marginparwidth 
  \@parboxrestore\tiny\raggedright #2}}
\catcode`\@=12

\AtBeginDocument{%
   \def\MR#1{}
}

\newcommand{\R}{\mathbb{R}}

\newcommand{\RR}{\mathbb{R}}
\newcommand{\HH}{\mathbb{H}}
\newcommand{\ZZ}{\mathbb{Z}}

\renewcommand{\P}{\mathcal P}
\newcommand{\A}{\mathcal A}
\newcommand{\W}{\mathcal W}
\newcommand{\vol}{{\rm vol}}

\newcommand{\lip}{{\rm Lip}}

\newcommand{\Oct}{{\rm Oct}}

\newcommand{\voct}{{v_{\rm oct}}}
\newcommand{\vtet}{{v_{\rm tet}}}

\def\co{\colon\thinspace}

\theoremstyle{plain}
\newtheorem{theorem}{Theorem}[section]
\newtheorem{corollary}[theorem]{Corollary}
\newtheorem{lemma}[theorem]{Lemma}

\newtheorem{conjecture}[theorem]{Conjecture}

\newtheorem*{namedtheorem}{\theoremname}
\newcommand{\theoremname}{testing}

\theoremstyle{definition}
\newtheorem{define}[theorem]{Definition}

\newtheorem{remark}[theorem]{Remark}

\title[Volume bounds for weaving knots]{Volume bounds for weaving knots}
\author[A.\ Champanerkar]{Abhijit Champanerkar}
\address{Department of Mathematics, College of Staten Island \& The Graduate Center, City University of New York, New York, NY}
\email{abhijit@math.csi.cuny.edu}
\author[I. \ Kofman]{Ilya Kofman}
\address{Department of Mathematics, College of Staten Island \& The Graduate Center, City University of New York, New York, NY}
\email{ikofman@math.csi.cuny.edu}
\author[J. \ Purcell]{Jessica S.\ Purcell}
\address{School of Mathematical Sciences, 9 Rainforest Walk, Monash University, Victoria 3800, Australia}
\email{jessica.purcell@monash.edu}

\begin{document}

\begin{abstract}
Weaving knots are alternating knots with the same projection as
torus knots, and were conjectured by X.-S.~Lin to be among the maximum
volume knots for fixed crossing number.  We provide the first
asymptotically sharp volume bounds for weaving knots, and we prove
that the infinite square weave is their geometric limit.
\end{abstract}

\maketitle

\section{Introduction}

The crossing number, or minimum number of crossings among all diagrams
of a knot, is one of the oldest knot invariants, and has been used to
study knots since the 19th century.  Since the 1980s, hyperbolic
volume has also been used to study and distinguish knots. We are
interested in the relationship between volume and crossing number.  On
the one hand, it is very easy to construct sequences of knots with
crossing number approaching infinity but bounded volume. For example,
start with a reduced alternating diagram of an alternating knot, and
add crossings by twisting two strands in a fixed twist region of the
diagram.
By work of J\o rgensen and Thurston, the volume of the resulting
sequence of knots is bounded by the volume of the link obtained by
augmenting the twist region (see \cite[page 120]{thurston:notes}). 
However, since reduced alternating diagrams realize the crossing number 
(see for example \cite{thistlethwaite}), the crossing number increases with 
the number of crossings in the twist region.

On the other hand, since there are only a finite number of knots with
bounded crossing number, among all such knots there must be one (or
more) with maximal volume. It is an open problem to determine the
maximal volume of all knots with bounded crossing number, and to find
the knots that realize the maximal volume per crossing number.

In this paper, we study a class of knots and links that are candidates
for those with the largest volume per crossing number: \emph{weaving knots}.
For these knots and links, we provide explicit, asymptotically sharp bounds on their volumes. We also prove that they converge geometrically to an infinite link complement which asymptotically maximizes volume per crossing number. Thus, while our methods cannot answer the question of which knots maximize volume per crossing number, they provide evidence that weaving knots are among those with largest volume
once the crossing number is bounded.

A {\em weaving knot} $W(p,q)$ is the alternating knot or link with the
same projection as the standard $p$--braid
$(\sigma_1\ldots\sigma_{p-1})^q$ projection of the torus knot or link
$T(p,q)$. Thus, the crossing number $c(W(p,q)) = q(p-1)$. For example,
$W(5,4)$ and $W(7,12)$ are shown in Figure~\ref{fig:Wpq}. By our
definition, weaving knots can include links with many components,
and throughout this paper, \emph{weaving knots} will denote both knots and links.

\begin{figure}[h]
\begin{tabular}{ccc}
\includegraphics[scale=0.15]{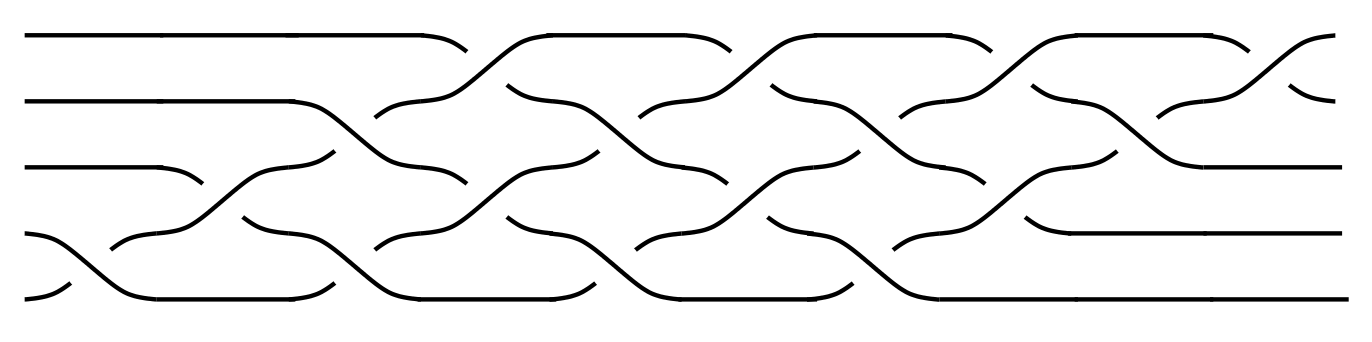} & \hspace*{0.5in} & \includegraphics[scale=0.22]{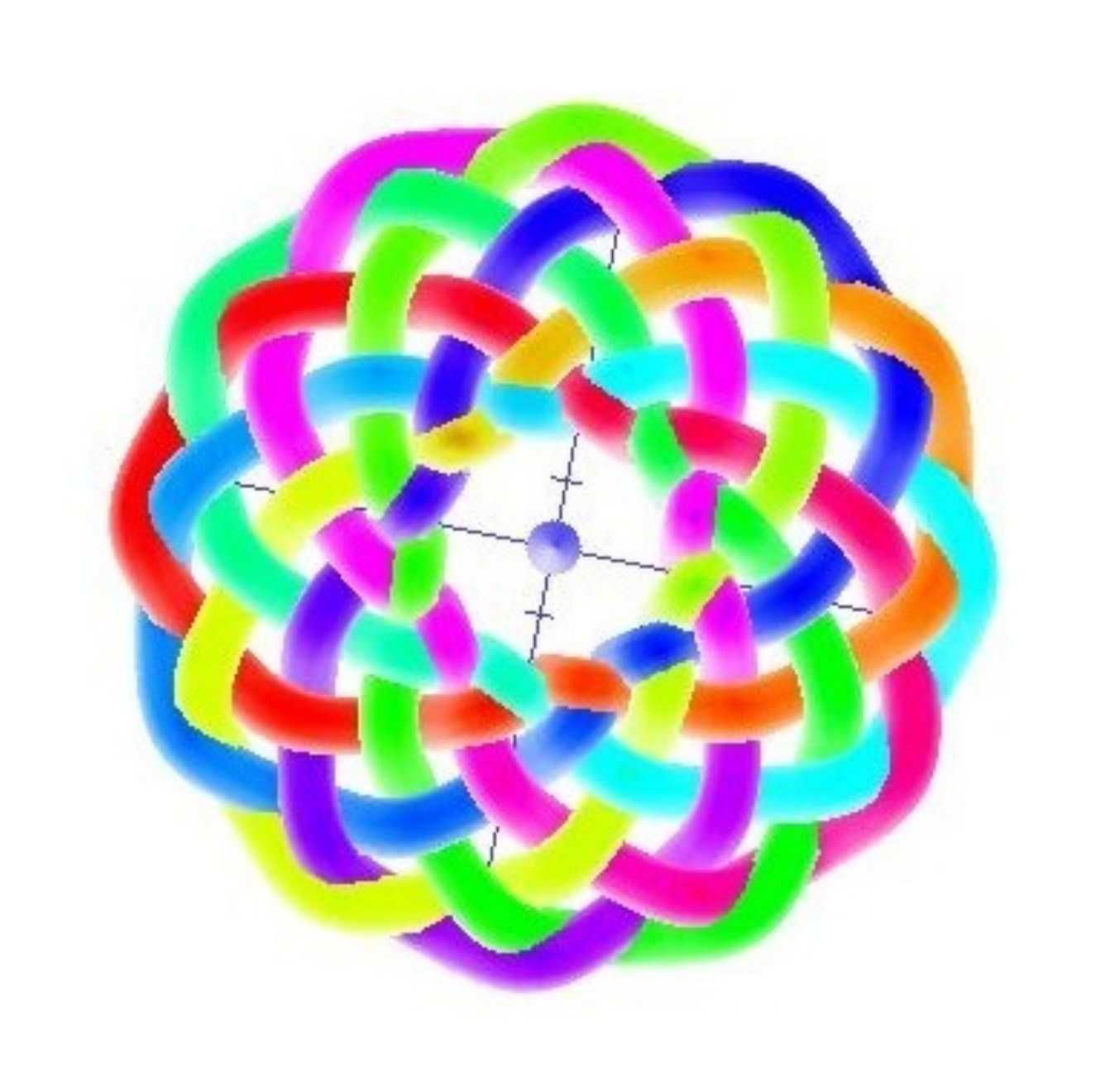}\\
(a) & \qquad & (b) \\
\end{tabular}
\caption{(a) $W(5,4)$ is the closure of this braid.
  (b) $W(7,12)$ figure modified from~\cite{pennock}. }
  \label{fig:Wpq}
\end{figure}

Xiao-Song~Lin suggested in the early 2000s that weaving knots would be among the knots with largest volume for fixed crossing number.
In fact, we checked that $W(5,4)$ has the second largest volume among all $1,701,936$ prime knots with $c(K) \leq 16$ (good guess!). These knots were classified by Hoste, Thistlethwaite, and Weeks \cite{HosteThistleWeeks}, and are available for further study, including volume computation, in the Knotscape~\cite{knotscape} census, or via SnapPy~\cite{snappy}.

It is a consequence of our main results in \cite{ckp:gmax} that weaving knots are {\em geometrically maximal}. That is, they satisfy:
\begin{equation}\label{eqn:VolLim}
  \lim_{p,q\to\infty}\frac{\vol(W(p,q))}{c(W(p,q))}= \voct,
\end{equation}
where $\voct \approx 3.66$ is the volume of a regular ideal octahedron, and $\vol(\cdot)$ and $c(\cdot)$ denote volume and crossing number, respectively. Moreover, it is known that for any link the {\em volume density} $\vol(K)/c(K)$ is always bounded above by $\voct$. 

What was not known is how to obtain sharp estimates on the volumes of
$W(p,q)$ in terms of $p$ and $q$ alone, which is needed to bound
volume for fixed crossing number.
We will say that volume bounds are \emph{asymptotically sharp} if the ratio of the lower and upper bounds approaches~$1$ as $p,\,q$ approach infinity.
Lackenby gave bounds on volumes of alternating knots and links~\cite{lackenby}. The upper bound was improved by Agol and D.~Thurston~\cite{lackenby}, and then again by Adams \cite{AdamsBound} and by Dasbach and Tsvietkova \cite{DasbachTsvietkova}. The lower bound was improved by Agol, Storm and Thurston~\cite{AgolStormThurston}.
However, these bounds are not asymptotically sharp, 
nor can they be used to establish the limit of
equation~\eqref{eqn:VolLim}. Our methods in~\cite{ckp:gmax} also fail
to give bounds on volumes of knots for fixed crossing number,
including $W(p,q)$. Thus, it seems that determining explicit,
asymptotically
sharp volume bounds in finite cases is harder and requires different methods than proving the asymptotic volume density $\vol(K_n)/c(K_n)$ for sequences
of links.

In this paper, for weaving knots $W(p,q)$ we provide asymptotically sharp, explicit bounds on volumes in terms of $p$ and $q$ alone.

\begin{theorem}\label{thm:lower-bound}
If $p\geq 3$ and $q\geq 7$, then
\[
\voct\, (p-2)\,q\, \left(1-\frac{(2\pi)^2}{q^2}\right)^{3/2} \: \leq \: \vol(W(p,q)) \: < \: (\voct\, (p-3) + 4\vtet)\,q.
\]
\end{theorem}
Here $\vtet \approx 1.01494$ is the volume of the regular ideal tetrahedron, and $\voct$ is the same as above.  Since $c(W(p,q)) = q\,(p-1)$, these bounds provide another proof of equation~\eqref{eqn:VolLim}.
In contrast, using \cite{AdamsBound, AgolStormThurston, DasbachTsvietkova, lackenby} the best current volume bounds for any knot or link $K$ with a prime alternating twist--reduced diagram with no bigons and $c(K)\geq 5$ crossings are
\[
\frac{\voct}{2}\, (c(K)-2) \: \leq \: \vol(K) \: \leq \: \voct\,(c(K)-5) + 4\vtet.
\]

The methods involved in proving Theorem~\ref{thm:lower-bound} are completely different than those used in \cite{ckp:gmax}, which relied on volume bounds via guts of 3--manifolds cut along essential surfaces as in~\cite{AgolStormThurston}.  Instead, the proof of Theorem~\ref{thm:lower-bound} involves explicit angle structures and the convexity of volume, as in~\cite{rivin}.

Moreover, applying these asymptotically sharp volume bounds for the links $W(p,q)$, we prove that their geometric structures converge, as follows.

The \emph{infinite square weave} $\W$ is defined to be the infinite alternating link with the square grid projection, as in Figure~\ref{fig:infweave}. In \cite{ckp:gmax}, we showed that there is a complete hyperbolic structure on $\RR^3-\W$ obtained by tessellating the manifold by regular ideal octahedra such that the volume density of $\W$ is exactly $\voct$.
The notion of a geometric limit is a standard way of expressing convergence of geometric structures on distinct manifolds, given in Definition~\ref{def:geomlimit} below.

\begin{figure}
 \includegraphics[scale=0.35]{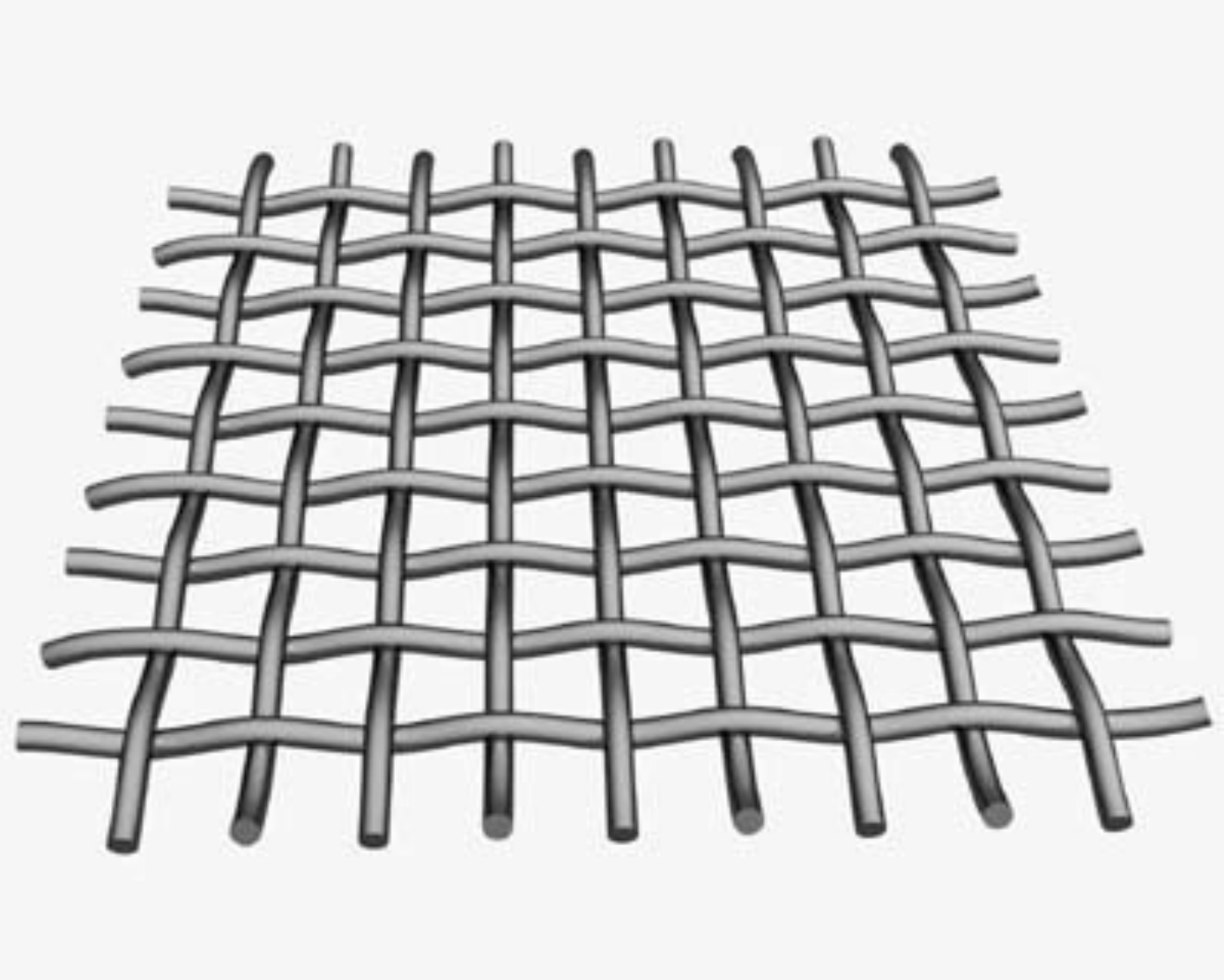}
\caption{The infinite alternating weave}
\label{fig:infweave}
\end{figure}

\begin{theorem}\label{thm:geolimit}
As $p,q \to \infty$, $S^3-W(p,q)$ approaches $\R^3-\W$ as a geometric limit.
\end{theorem}

Proving that a class of knots or links approaches $\R^3-\W$ as a geometric limit seems to be difficult. For example, in~\cite{ckp:gmax} we showed that many families of knots $K_n$ with diagrams approaching that of $\W$, in an appropriate sense, satisfy $\vol(K_n)/c(K_n) \to \voct$. However, it is unknown whether their complements $S^3-K_n$ approach $\R^3-\W$ as a geometric limit, and the proof in~\cite{ckp:gmax} does not give this information. 
Theorem~\ref{thm:geolimit} provides the result for $W(p,q)$.

It is an interesting fact that every knot and link can be obtained by
changing some crossings of $W(p,q)$ for some $p,\,q$.
This was proved for standard diagrams of torus knots and links by Manturov~\cite{Manturov}, so the same result holds for weaving knots as well.
We conjecture that the upper volume bound in Theorem \ref{thm:lower-bound} applies to any knot or link
obtained by changing crossings of $W(p,q)$.
This conjectured upper bound would give better volume bounds in certain cases than the general bounds mentioned above,
but more significantly, this conjecture is a special case (and provides a test) of the following conjecture, which appears in \cite{ckp:gmax}. 

\begin{conjecture}\label{conj:alt}
Let $K$ be an alternating hyperbolic knot or link, and $K'$ be obtained by changing any proper subset of crossings of $K$.  Then $\displaystyle \vol(K') < \vol(K)$.
\end{conjecture}

\subsection{Acknowledgements}
We thank Craig Hodgson for helpful conversations. We also thank an anonymous referee for detailed comments, which have improved the clarity and accuracy of this paper. The first two authors acknowledge support by the Simons Foundation and PSC-CUNY. The third author acknowledges support by the National Science Foundation under grants number DMS--1252687 and DMS-1128155.

\section{Triangulation of weaving knots}\label{sec:triangulation}

Consider the weaving knot $W(p,q)$ as a closed $p$--braid. Let $B$
denote the braid axis. In this section, we describe a decomposition of
$S^3-(W(p,q)\cup B)$ into ideal tetrahedra and octahedra. This leads
to our upper bound on volume, obtained in this section.  In
Section~\ref{sec:lowerbounds} we will use this decomposition to prove
the lower bound as well.

Let $p\geq 3$. Note that the complement of $W(p,q)$ in $S^3$ with the braid axis also removed is a $q$--fold cover of the complement of $W(p,1)$ and its braid axis.

\begin{lemma}\label{lemma:polyhedra}
Let $B$ denote the braid axis of $W(p,1)$.  Then $S^3-(W(p,1)\cup B)$ admits an ideal polyhedral decomposition $\P$ with four ideal tetrahedra and $p-3$ ideal octahedra.

Moreover, a meridian for the braid axis runs over exactly one side of one of the ideal tetrahedra.  The polyhedra give a polygonal decomposition of the boundary of a horoball neighborhood of the braid axis, with a fundamental region consisting of four triangles and $2(p-3)$ quadrilaterals, as shown in Figure~\ref{fig:cusp-triang}.
\end{lemma}

\begin{figure}
\includegraphics[scale=0.5]{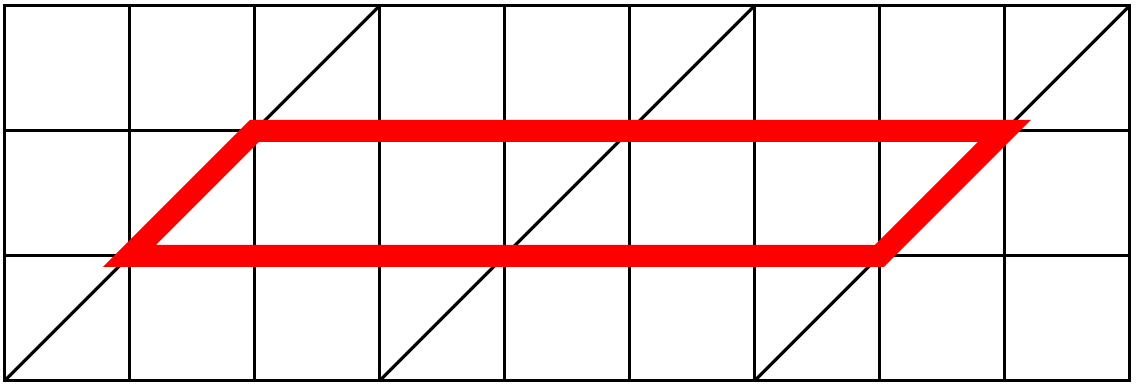}
\caption{Polygonal decomposition of cusp corresponding to braid axis.  A fundamental region consists of four triangles and $2(p-3)$ quads.  The example shown here is $p=5$. The black diagonal is the preimage of a meridian.}
\label{fig:cusp-triang}
\end{figure}

\begin{proof}
Consider the standard diagram of $W(p,1)$ in a projection plane, which $B$ intersects in two points.
Obtain an ideal polyhedral decomposition as follows.  First, for every
crossing of the $W(p,1)$ diagram, take an ideal edge, the
\emph{crossing arc}, running from the knot strand at the top of the
crossing to the knot strand at the bottom.  This subdivides the
projection plane into two triangles and $p-3$ quadrilaterals that
correspond to the regions of the link projection.  This is shown in
Figure~\ref{fig:polyhedra} (left) when $p=5$.  In the figure, note
that the four dotted red edges shown at each crossing are homotopic to
the crossing arc.

\begin{figure}
  \includegraphics[scale=0.55]{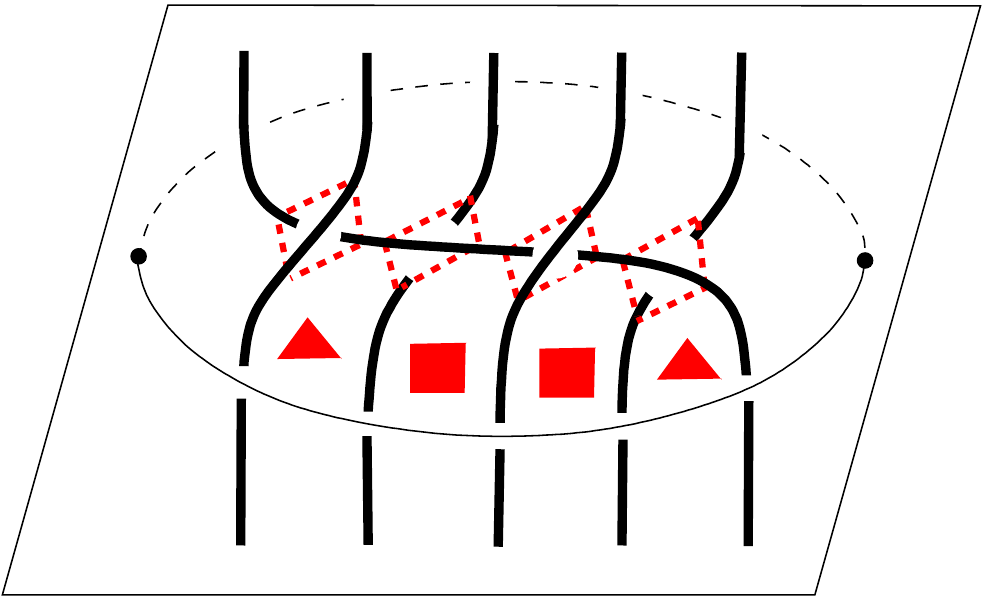}
  \includegraphics[scale=0.12]{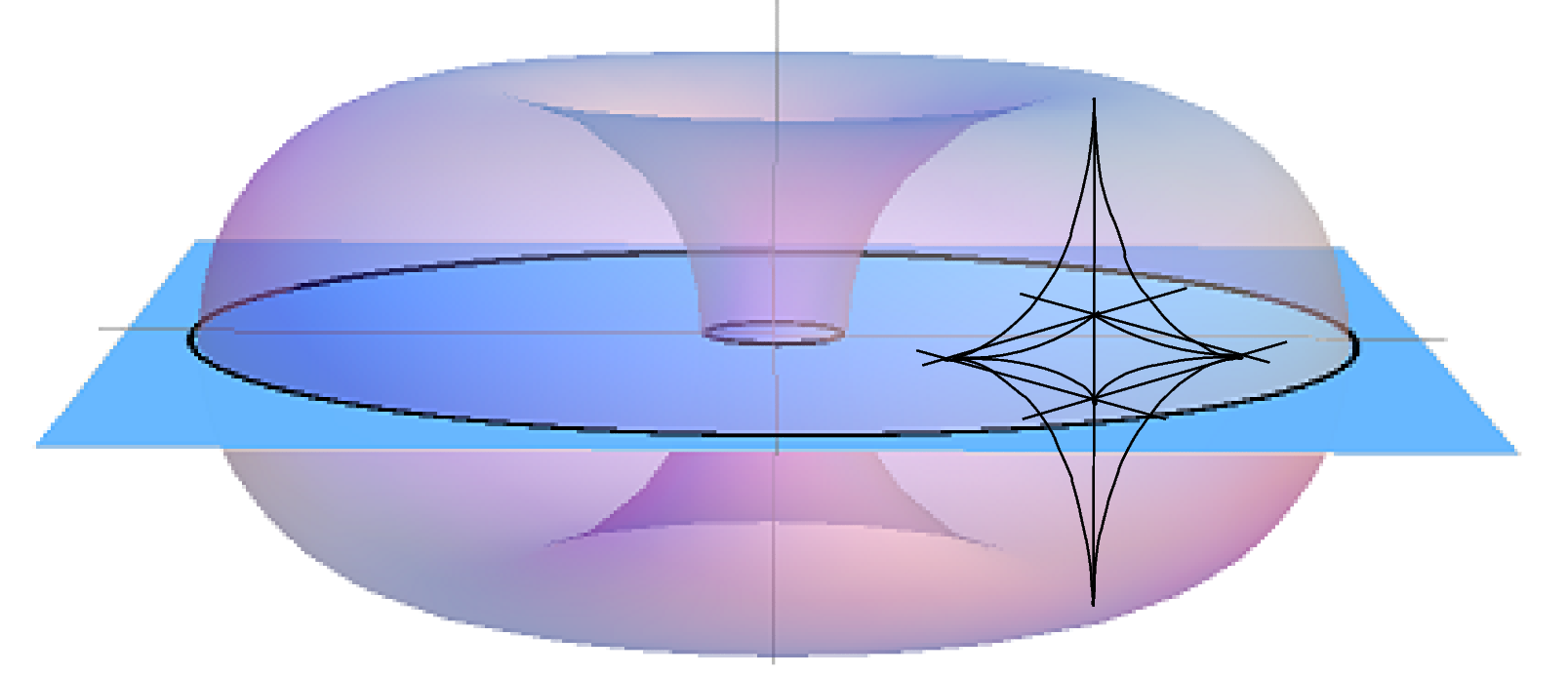}
  \caption{Left: Dividing projection plane into triangles and quadrilaterals.  Right:  Coning to ideal octahedra and tetrahedra, figure modified from \cite{torusfig}.} 
  \label{fig:polyhedra}
\end{figure}

Now for each quadrilateral on the projection plane, add four edges
between ideal vertices above the projection plane and four below, as
follows.  Those edges above the projection plane run vertically from
the strand of $W(p,1)$ corresponding to an ideal vertex of the
quadrilateral to the braid axis $B$.  Those edges below the projection
plane also run from strands of $W(p,1)$ corresponding to ideal
vertices of the quadrilateral, only now they run below the projection
plane to $B$. These edges bound eight triangles, as follows.  Four of
the triangles lie above the projection plane, with two sides running
from a strand of $W(p,1)$ to $B$ and the third in the projection
plane, connecting two vertices of the quadrilateral.
The other four lie below, again each with two edges running from
strands of $W(p,1)$ to $B$ and one edge on the quadrilateral.  The
eight triangles together bound a single octahedron.  This is shown in
Figure \ref{fig:polyhedra} (right).  Note there are $p-3$ such
octahedra coming from the $p-3$ quadrilaterals on the projection
plane.

As for the tetrahedra, these come from the triangular regions on the projection plane.  As above, draw three ideal edges above the projection plane and three below.  Each ideal edge runs from a strand of $W(p,1)$ corresponding to an ideal vertex of the triangle.  For each ideal vertex, one edge runs above the projection plane to $B$ and the other runs below to $B$.  Again we form six ideal triangles per triangular region on the projection plane.  This triangular region along with the ideal triangles above the projection plane bounds one of the four tetrahedra.  The triangular region along with ideal triangles below the projection plane bounds another.  There are two more coming from the ideal triangles above and below the projection plane for the other region.

There are also two exterior regions of the diagram, which meet $B$. We do not add tetrahedra or octahedra to these regions, because faces of tetrahedra already encountered above will be glued in these two regions. This can be seen as follows. Consider the region meeting $B$ on the left side of the diagram in Figure~\ref{fig:polyhedra} (left). In this region, there is only one edge from $W(p,1)$ to $W(p,1)$. On either side of this edge, add an edge connecting $W(p,1)$ to $B$, forming a triangle, which is sketched in Figure~\ref{fig:GluingPolyhedra}, left. The triangle is labeled $T$. It lies in the projection plane, with one edge labeled $1$, one labeled $2$, and the third edge dotted. 

\begin{figure}
  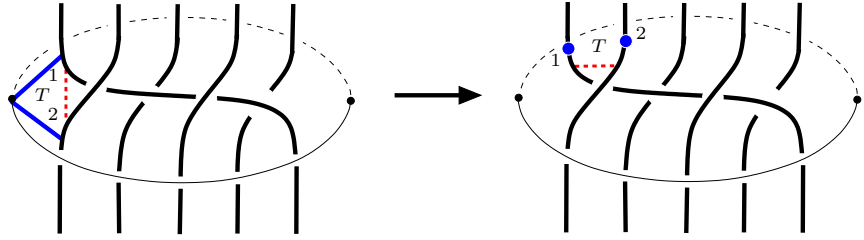
  \caption{A triangular face $T$ in the exterior region is isotopic to a face of the tetrahedron obtained from the first triangle of Figure~\ref{fig:polyhedra} (left).}
  \label{fig:GluingPolyhedra}
\end{figure}

Note that the two edges labeled 1 and 2 of the triangle $T$ have one
of their end points on $B$ and the other on the strand of $W(p,1)$
which connects top to bottom as the braid closes up, on the far left
side of Figure \ref{fig:GluingPolyhedra}. Hence they are isotopic.
Thus the triangle $T$ completely fills the region of the projection
plane on the left of the diagram.  However, note if we isotope the two
edges from $W(p,1)$ to $B$ to be vertical, and slide the one on top
slightly past the crossing, we see that the triangle is isotopic to a
face of the tetrahedron coming from the top left of the
diagram. Similarly, the triangle is isotopic to a face of the
tetrahedron on the bottom left. Hence these two tetrahedra are glued,
top to bottom, along these triangular faces. Similarly, tetrahedra on
top and bottom on the right are glued along an exterior triangular
region in the projection plane.

Now we claim the collection of tetrahedra and octahedra glue to give
the link complement, and hence form the claimed polyhedral
decomposition.  Note that tetrahedra are glued in pairs across the
projection plane in the triangular regions of the diagram.  In
addition, as noted above, on either side of the diagram, one
triangular face of a tetrahedron above the projection plane with two
edges running to $B$ is identified to one triangular face of a
tetrahedron below with two edges running to $B$. The identification
maps both triangles to one in the exterior region of the diagram, as
shown in Figure~\ref{fig:GluingPolyhedra}. All other triangular and
quadrilateral faces are identified by obvious homotopies of the edges
and faces. This concludes the proof that these tetrahedra and
octahedra form the desired polyhedral decomposition.

Finally, to see that the cusp cross section of $B$ meets the polyhedra
as claimed, we need to step through the gluings of the portions of
polyhedra meeting $B$.  As noted above, where $B$ meets the projection
plane, there is a single triangular face $T$ of two tetrahedra, as in
Figure~\ref{fig:GluingPolyhedra}. 
The two edges of the triangle $T$ labeled 1 and 2 are isotopic in $S^3-(W(p, 1)\cup B)$, where the isotopy takes the ideal vertex of edge 1 on $W(p,1)$ around the braid closure to the ideal vertex of edge 2 on $W(p,1)$.  The
other ideal vertex of edge 1 follows a meridian of $B$ under this isotopy.
Hence a regular neighborhood of the ideal vertex of $T$ lying on $B$ traces an entire meridian of $B$. Thus a meridian of $B$ is given by the intersection of exactly one face (namely $T$) of one of the ideal tetrahedra
with a cusp neighborhood of $B$.
Now, two tetrahedra, one from above the plane of projection, and one
from below, are glued along that face.  The other two faces of the
tetrahedron above the projection plane are glued to two distinct sides
of the octahedron directly adjacent, above the projection plane. The
remaining two sides of this octahedron above the projection plane are
glued to two distinct sides of the next adjacent octahedron, above the
projection plane, and so on, until we meet the tetrahedron above the
projection plane on the opposite end of $W(p,1)$, which is glued below
the projection plane. Now following the same arguments, we see the
triangles and quadrilaterals repeated below the projection plane,
until we meet up with the original tetrahedron. Hence the cusp shape
is as shown in Figure \ref{fig:cusp-triang}.
\end{proof}

\begin{corollary}\label{cor:upper-bound}
For $p\geq 3$, the volume of $W(p,q)$ is
less than $(4\vtet + (p-3)\,\voct)\, q$.
\end{corollary}

\begin{proof}
For any positive integer $q$, and any $p\geq 3$, we claim that $S^3-(W(p,q)\cup B)$ is hyperbolic. For fixed $p\geq 3$ and $q_0$ large, say $q_0\geq 6$, the diagram of $W(p,q_0)$ is a reduced diagram of a prime alternating link that is not a 2-braid, so is hyperbolic~\cite{menasco}. Then recent work of Adams implies that when we remove the braid axis from the complement, the resulting link remains hyperbolic~\cite[Theorem~2.1]{adams:genaug}. Since $S^3-(W(p,q_0)\cup B)$ is a finite cover of $S^3-(W(p,1)\cup B)$, the latter manifold is also hyperbolic. Hence for any positive integer $q$, the cover $S^3-(W(p,q)\cup B)$ will also be hyperbolic.

Now, the hyperbolic manifold $S^3-(W(p,q)\cup B)$ has a decomposition
into ideal tetrahedra and octahedra.  The maximal volume of a
hyperbolic ideal tetrahedron is $\vtet$, the volume of a regular ideal
tetrahedron. The maximal volume of a hyperbolic ideal octahedron is at
most $\voct$, the volume of a regular ideal octahedron. The result now
follows immediately from the first part of
Lemma~\ref{lemma:polyhedra}, and the fact that volume strictly
decreases under Dehn filling \cite{thurston:notes}.
\end{proof}

\subsection{Weaving knots with three strands}\label{3strand}
The case when $p=3$ is particularly nice geometrically, and so we treat it separately in this section.

\begin{theorem}\label{thm:3-strand}
If $p=3$ then the upper bound in Corollary~\ref{cor:upper-bound} is achieved exactly by the volume of $S^3-(W(3,q)\cup B),$ where $B$ denotes the braid axis. That is, 
$$ \vol(W(3,q)\cup B) = 4\,q\,\vtet.$$
\end{theorem}

\begin{proof}
Since the complement of $W(3, q)\cup B$ in $S^3$ is a $q$--fold cover of the complement of $W(3, 1) \cup B$, it is enough to prove the statement for $q=1$. 

We proceed as in the proof of Lemma \ref{lemma:polyhedra}. If $p=3$, then the projection plane of $W(3,1)$ is divided into two triangles; see Figure \ref{fig:3-strand}. This gives four tetrahedra, two each on the top and bottom.  The edges and faces on the top tetrahedra are glued to those of the bottom tetrahedra across the projection plane for the same reason as in the proof Lemma~\ref{lemma:polyhedra}. 

\begin{figure}[h]
  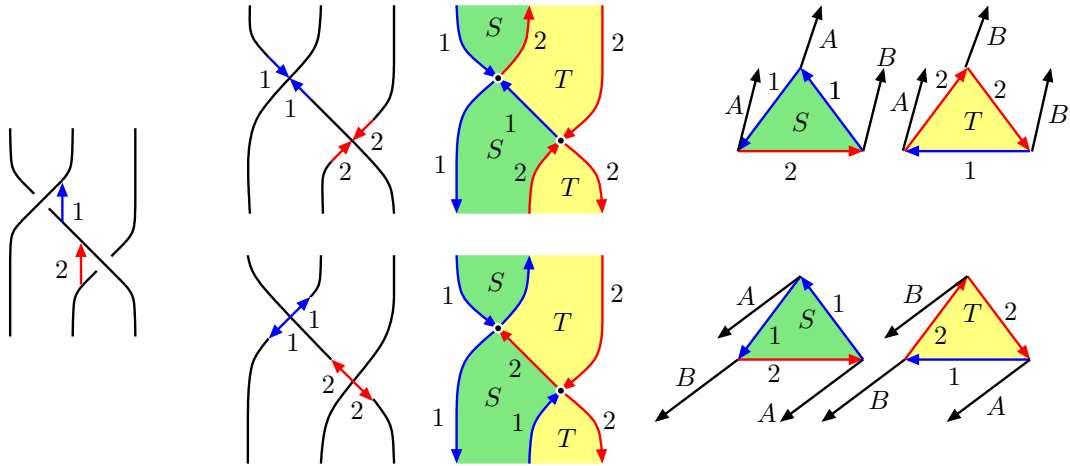
 \caption{The tetrahedral decomposition of $S^3-(W(3,1)\cup B)$. } 
  \label{fig:3-strand}
\end{figure}

Thus the tetrahedra are glued as shown in Figure~\ref{fig:3-strand}.  The top figures indicate the top tetrahedra and the bottom figures indicate the bottom tetrahedra. The crossing edges are labelled by numbers and the edges from the knot to the braid axis are labelled by letters. The two triangles in the projection plane are labelled $S$ and $T$. Edges and faces are glued as shown. 

In this case, all edges of the polyhedra
are $6$--valent.  We set all tetrahedra to be regular ideal tetrahedra, and obtain a solution to the gluing equations.
Since all links of tetrahedra are equilateral triangles, they are all similar, and all edges of any triangle are scaled by the same factor under dilations. Hence, the holonomy for every loop in the cusp has to expand and contract by the same factor (i.e.\ it is scaled by unity), and so it is a Euclidean isometry. This implies that the regular ideal tetrahedra are also a solution to the completeness equations. Thus this is a geometric triangulation giving  the complete structure with volume $4 \vtet$. 
\end{proof}

\begin{remark}
Since the volumes of $S^3-(W(3,q)\cup B)$ are integer
multiples of $\vtet$, we investigated their commensurability with the complement of the figure--8 knot, which is $W(3,2)$. Using SnapPy~\cite{snappy}, we verified that $S^3-(W(3,2)\cup B)$ is
a 4--fold cover of $S^3-W(3,2)$.  Thus, the figure--8 knot complement is covered by its braid complement with the axis removed!  Some other interesting links also appear in this commensurablity class, as illustrated in
Figure~\ref{fig:3-strand-comm}.
\end{remark}

\begin{figure}
  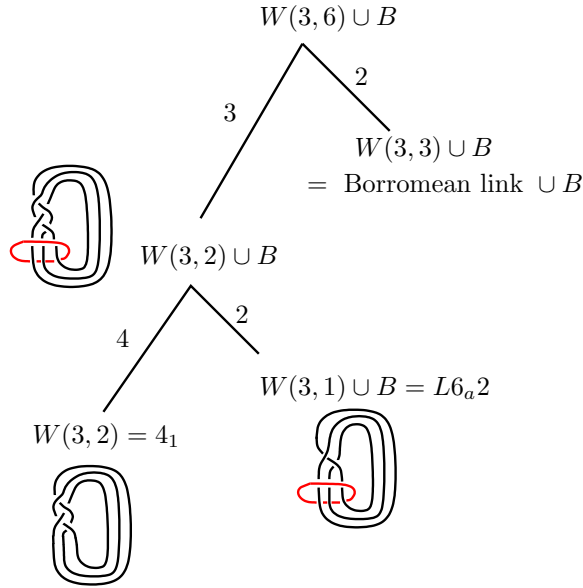
\caption{The complement of the figure--8 knot and its braid axis, $S^3-(W(3,2)\cup B)$, is a 4--fold cover of the figure--8 knot complement, $S^3-W(3,2)$.}
\label{fig:3-strand-comm}
\end{figure}

\section{Angle structures and lower volume bounds}\label{sec:lowerbounds}

In this section we find lower bounds on volumes of weaving knots. To do so, we use angle structures on the manifolds $S^3-(W(p,q)\cup B)$.

\begin{define}\label{def:AngleStruct}
Given an ideal triangulation $\{\Delta_i\}$ of a 3--manifold, an \emph{angle structure} is a choice of three angles $(x_i, y_i, z_i) \in (0, \pi)^3$ for each tetrahedron $\Delta_i$, assigned to three edges of $\Delta_i$ meeting in a vertex, such that
\begin{enumerate}
\item\label{item:sumtopi} $x_i + y_i + z_i = \pi$;
\item\label{item:oppedge} the edge opposite the edge assigned angle $x_i$ in $\Delta_i$ is also assigned angle $x_i$, and similarly for $y_i$ and $z_i$; and
\item angles about any edge add to $2\pi$.
\end{enumerate}
\end{define}

For any tetrahedron $\Delta_i$ and angle assignment $(x_i, y_i, z_i)$ satisfying \eqref{item:sumtopi} and \eqref{item:oppedge} above, there exists a unique hyperbolic ideal tetrahedron with the same dihedral angles.  
The volume of this hyperbolic ideal tetrahedron can be computed from $(x_i, y_i, z_i)$.  We do not need the exact formula for our purposes.  However, given an angle structure on a triangulation $\{\Delta_i\}$, we can compute the corresponding volume by summing all volumes of ideal tetrahedra with that angle assignment.

\begin{lemma}\label{lemma:angle-struct}
For $p>3$, the manifold $S^3-(W(p,1)\cup B)$ admits an angle structure with volume $\voct\,(p-2)$.
\end{lemma}

\begin{proof}
We will take our ideal polyhedral decomposition of $S^3-(W(p,1)\cup B)$ from Lemma~\ref{lemma:polyhedra} and turn it into an ideal triangulation by stellating the octahedra, splitting each of them into four ideal tetrahedra. More precisely, this is done by adding an edge running from the ideal vertex on the braid axis above the plane of projection, through the plane of projection to the ideal vertex on the braid axis below the plane of projection. Using this ideal edge, the octahedron is split into four tetrahedra.

Now obtain an angle structure on this triangulation as follows. First, 
assign to each edge in an octahedron (edges that existed before stellating) the angle $\pi/2$.  As for the four tetrahedra, assign angles $\pi/4$, $\pi/4$, and $\pi/2$ to each, such that pairs of the tetrahedra glue into squares in the cusp neighborhood of the braid axis. 
See Figure~\ref{fig:cusp-triang}. When we stellate, assign angle structures to the four new tetrahedra coming from the octahedra in the obvious way, namely, on each new tetrahedron the ideal edge through the plane of projection is given angle $\pi/2$, and the other two edges meeting that edge in an ideal vertex are labeled $\pi/4$. 
With these angles, items (1) and (2) from Definition~\ref{def:AngleStruct} are satisfied for the tetrahedra. We need to check item (3). 

Note the angle sum around each new edge in the stellated octahedra is $2\pi$, so we only need to consider edges coming from the original tetrahedra and octahedra of the polyhedral decomposition, and the angle sums around them. Consider first the ideal edges with one endpoint on $W(p,1)$ and one on the braid axis.  These correspond to vertices of the polygonal decomposition of the braid axis illustrated in Figure~\ref{fig:cusp-triang}. Note that many of these edges meet exactly four ideal octahedra, hence the angle sum around them is $2\pi$. Any such edge that meets an original tetrahedron either meets three other ideal octahedra and the angle in the tetrahedron is $\pi/2$, so the total angle sum is $2\pi$, or it is identified to four edges of tetrahedra with angle $\pi/4$, and two octahedra. Hence the angle sum around it is $2\pi$.

Finally consider the angle sum around edges which run from $W(p,1)$ to $W(p,1)$.  These arise from crossings in the diagram of $W(p,1)$. The first two crossings on the left side and the last two crossings on the right side give rise to ideal edges bordering (some) tetrahedra.  The others (for $p>4$) border only octahedra, and exactly four such octahedra, hence the angle sum for those is $2\pi$.  So we need only consider the edges arising from two crossings on the far left and two crossings on the far right.  We consider those on the far left; the argument for the far right is identical.

Label the edge at the first crossing on the left $1$, and label that of the second $2$. See Figure~\ref{fig:edges12}. The two tetrahedra arising on the far left have edges glued as shown on the left of Figure~\ref{fig:edge-labels-tet}, and the adjacent octahedron has edges glued as on the right of that figure. We label the tetrahedra $T_1$ and $T_1'$.
    
\begin{figure}
  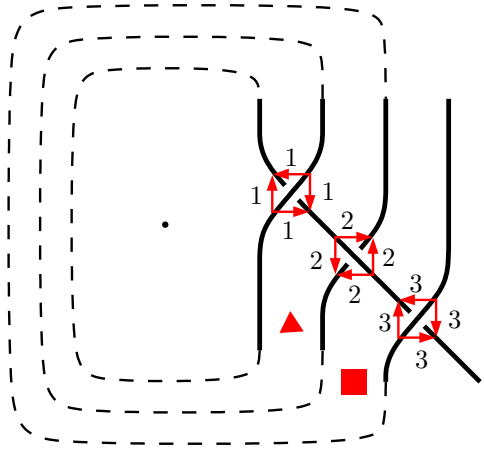
\caption{Each edge $1$ and $2$ is a part of two tetrahedra arising from
    the triangle, and an octahedron arising from the square as shown in
    Figure \ref{fig:edge-labels-tet}. The braid axis is shown in the center.
  }
  \label{fig:edges12}
\end{figure}

\begin{figure}
  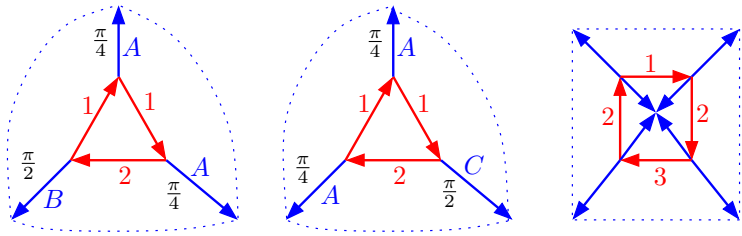
  \caption{Edges are glued as shown in figure. From left to right, shown are tetrahedra $T_1$, $T_1'$, and adjacent octahedron.}
  \label{fig:edge-labels-tet}
\end{figure}

Note that the edge labeled $1$ in the figure is glued four times in
tetrahedra, twice in $T_1$ and twice in $T_1'$, and once in an
octahedron.  However, note that in the tetrahedra it is assigned
different angle measurements.  In particular, in $T_1$, the edge
labeled $1$, which is opposite the edge labeled $b$, must have
dihedral angle $\pi/2$, because that is the angle on the edge labeled
$b$.  The other edge of $T_1$ labeled $1$ must have dihedral angle
$\pi/4$.  Similarly for $T_1'$. Thus the angle sum around the edge
labeled $1$ is $2\pi$.

In both $T_1$ and $T_1'$, the edge labeled $2$ has angle $\pi/4$. Since the edge labeled $2$ is also glued to two edges in one octahedron, and one edge in another, the total angle sum around that edge is also $2\pi$. Hence this gives an angle sum as claimed. This concludes the proof that $S^3-(W(p,1)\cup B)$ admits an angle structure.

The volume estimate comes from the fact that a regular ideal octahedron has volume $\voct$.  Moreover, four ideal tetrahedra, each with angles $\pi/2, \pi/4, \pi/4$, can be glued to give an ideal octahedron, hence each such tetrahedron has volume $\voct/4$.  We have $p-3$ octahedra and four such tetrahedra, and hence the corresponding volume is $(p-2)\,\voct$.  
\end{proof}

\begin{lemma}\label{lemma:stellatedoct}
Let $P$ be an ideal polyhedron obtained by coning to $\pm\infty$ from any ideal quadrilateral in the projection plane.  Then for any angle structure on $P$, the volume of that angle structure $\vol(P)$ satisfies
$\vol(P)\leq \voct$, the volume of the regular ideal octahedron.
\end{lemma}
\begin{proof}
Suppose the volume for some angle structure is strictly greater than $\voct$.  The dihedral angles on the exterior of $P$ give a dihedral angle assignment $\Delta$ to $P$, and so in the terminology of Rivin \cite{rivin}, $A(P, \Delta)$ is nonempty.  By Theorem 6.13 of that paper, there is a unique complete structure with angle assignment $\Delta$, and the proof of Theorem 6.16 of \cite{rivin} shows that the complete structure occurs exactly when the volume is maximized over $A(P, \Delta)$.  Hence the volume of our angle structure is at most the volume of the complete hyperbolic structure on $P$ with angle assignment $\Delta$.  

On the other hand, for complete hyperbolic structures on $P$, it is known that the volume is maximized in the regular case, and thus the volume is strictly less than the volume of a regular ideal octahedron.  The proof of this fact is given, for example, in Theorem 10.4.8 and the proof of Theorem 10.4.7 in \cite{ratcliffe}.  This is a contradiction.
\end{proof}

Now consider the space $\A(\P)$ of angle structures on the ideal triangulation $\P$ for $S^3-(W(p,1)\cup B)$.
Given an angle structure on a tetrahedron, there is a unique ideal tetrahedron in $\HH^3$ with the same dihedral angles, and it admits a volume. Thus, there is a volume function $\vol\co \A(\P)\to \RR$ given by summing all volumes of tetrahedra.
It is known that the volume function is concave down on the space of angle structures
(see \cite{futer-gueritaud:angles, rivin}).
If we can show that the critical point of the function $\vol$ lies in the interior of $\A(\P)$,
then work of Casson and Rivin will imply that any angle structure is a lower bound on the hyperbolic volume of the manifold \cite{rivin}. Hence, we must study the behavior of the function $\vol$ on the boundary of $\A(\P)$. From Definition~\ref{def:AngleStruct}, we see that the space of angle structures is a linear subspace of $(0,\pi)^{3n}$, where $n$ is the total number of tetrahedra in the triangulation. Thus, the boundary of $\A(\P)$ lies on the boundary of $[0,\pi]^{3n}$. 
(For an excellent exposition on angle structures and volume, see \cite{futer-gueritaud:angles}.) 

\begin{lemma}\label{lemma:crit-at-interior}
The critical point for $\vol \co \A(\P)\to \RR$ is in the interior of the space $\A$ of angle structures on $\P$.
\end{lemma}

\begin{proof}
We will show that the volume function takes values strictly smaller on the boundary of $\A(\P)$ than at any point in the interior.  Therefore, it will follow that the maximum occurs in the interior of $\A(\P)$.

Suppose we have a point $X$ on the boundary of $\A(\P)$ that maximizes volume.  Because the point is on the boundary, there must be at least one triangle $\Delta$  with angles $(x_0, y_0, z_0)$ where one of $x_0$, $y_0$, and $z_0$ equals zero or $\pi$.
If one equals $\pi$, then condition (1) of Definition~\ref{def:AngleStruct} implies another equals zero. So we assume one of $x_0$, $y_0$, or $z_0$ equals zero.
A proposition of Gu\'eritaud, \cite[Proposition~7.1]{gf:once-punct-torus}, implies that if one of $x_0$, $y_0$, $z_0$ is zero, then another is $\pi$ and the third is also zero.  (The proposition is stated for once--punctured torus bundles in \cite{gf:once-punct-torus}, but only relies on the formula for volume of a single ideal tetrahedron, \cite[Proposition~6.1]{gf:once-punct-torus}.)

A tetrahedron with angles $0$, $0$, and $\pi$ is a flattened tetrahedron, and contributes nothing to volume.  We consider which tetrahedra might be flattened.

Let $\P_0$ be the original polyhedral decomposition described in the proof of Lemma~\ref{lemma:polyhedra}.  Suppose first that we have flattened one of the four tetrahedra which came from tetrahedra in $\P_0$.  Then the maximal volume we can obtain from these four tetrahedra is at most $3\,\vtet$, which is strictly less than $\voct$, which is the volume we obtain from these four tetrahedra from the angle structure of Lemma~\ref{lemma:angle-struct}.  Thus, by Lemma~\ref{lemma:stellatedoct}, the maximal volume we can obtain from any such angle structure is $3\, \vtet + (p-3)\voct < (p-2)\voct$.  Since the volume on the right is realized by an angle structure in the interior by Lemma~\ref{lemma:angle-struct}, the maximum of the volume cannot occur at such a point of the boundary.

Now suppose one of the four tetrahedra coming from an octahedron is flattened.  Then the remaining three tetrahedra can have volume at most $3\, \vtet < \voct$.  Thus the volume of such a structure can be at most $4\, \vtet + 3\, \vtet + (p-4)\, \voct$, where the first term comes from the maximum volume of the four tetrahedra in $\P_0$, the second from the maximum volume of the stellated octahedron with one flat tetrahedron, and the last term from the maximal volume of the remaining ideal octahedra.  Because $7\,\vtet < 2\, \voct$, the volume of this structure is still strictly less than that of Lemma~\ref{lemma:angle-struct}.

Therefore, there does not exist $X$ on the boundary of the space of angle structures that maximizes volume. 
\end{proof}

\begin{theorem}
\label{thm:WUaxis}
If $p>3$, then
\[ \voct\,(p-2)q \leq \vol(W(p,q)\cup B) %
<
(\voct\,(p-3)+4\vtet)q.\]
If $p=3$, then $\displaystyle \vol(W(3,q)\cup B) = 4q\,\vtet$.
\end{theorem}

\begin{proof}
Theorem~\ref{thm:3-strand} provides the $p=3$ case.

For $p>3$, Casson and Rivin showed that if the critical point for the volume is in the interior of the space of angle structures, then the maximal volume angle structure is realized by the actual hyperbolic structure \cite{rivin}. By Lemma~\ref{lemma:crit-at-interior}, the critical point for volume is in the interior of the space of angle structures.  By Lemma~\ref{lemma:angle-struct}, the volume of one particular angle structure is $\voct\,(p-2)q$.  So the maximal volume must be at least this.
The upper bound is from Corollary~\ref{cor:upper-bound}.
\end{proof}

Since $S^3-W(p,q)$ is obtained from $S^3-(W(p,q)\cup B)$ by Dehn filling along a meridian slope, we obtain geometric information on $W(p,q)$ given information on the geometry of this slope. In particular, the boundary of any embedded horoball neighborhood of the cusp $B$ inherits a Euclidean metric, and a closed geodesic representing the meridian inherits a length in this metric, called the \emph{slope length}. Note that slope length depends on choice of horoball neighborhood of $B$. Throughout, we will choose the horoball neighborhood of $B$ to be maximal, meaning it is tangent to itself. 

\begin{lemma}\label{lemma:meridian-length}
The length of a meridian of the braid axis is at least $q$.
\end{lemma}

\begin{proof}
A meridian of the braid axis of $W(p,q)$ is a $q$--fold cover of a meridian of the braid axis of $W(p,1)$.
In a maximal cusp neighborhood,
the meridian of the braid axis of $W(p,1)$ must have length at least one (see \cite{adams:waist, thurston:notes}).
Hence the meridian of the braid axis of $W(p,q)$ has length at least $q$.
\end{proof}

We can now prove our main result on volumes of weaving knots:

\begin{proof}[Proof of Theorem \ref{thm:lower-bound}]
The upper bound is from Corollary~\ref{cor:upper-bound}.

As for the lower bound, the manifold $S^3-W(p,q)$ is obtained by Dehn filling the meridian on the braid axis of $S^3-(W(p,q)\cup B)$.
When $q>6$, Lemma~\ref{lemma:meridian-length} implies that the meridian of the braid axis has length greater than $2\pi$, and so \cite[Theorem~1.1]{fkp:volume} will apply. Combining
\cite[Theorem~1.1]{fkp:volume}
with Theorem~\ref{thm:WUaxis} implies,
for $p>3$,
\[
\left(1 -\left(\frac{2\pi}{q}\right)^2\right)^{3/2} ( (p-2)\, q \, \voct ) \leq \vol(S^3-W(p,q)).
\]
For $p=3$,
\[
\left(1 -\left(\frac{2\pi}{q}\right)^2\right)^{3/2} ( 4\, q \, \vtet ) \leq \vol(S^3-W(3,q)).
\]
Since $\voct < 4\vtet$, this equation gives the desired lower bound when $p=3$. Thus we have the result for all $p\geq 3$.
\end{proof}

\begin{corollary}\label{cor:v3_links}
The links $K_n = W(3,n)$ satisfy $\displaystyle \lim_{n\to\infty}\frac{\vol(K_n)}{c(K_n)}=2\vtet.$
\end{corollary}
\begin{proof}
By Theorem~\ref{thm:3-strand} and the same argument as above, for $q>6$ we have
\[
\left(1 -\left(\frac{2\pi}{q}\right)^2\right)^{3/2} (4q\,\vtet) \leq \vol(S^3-W(3,q))
<
4q\,\vtet. \qedhere
\]
\end{proof}

\section{Geometric convergence of weaving knots}\label{sec:geomconv}

In this section, we will prove Theorem~\ref{thm:geolimit}, which states that as $p,q\to\infty$, the manifold $S^3-W(p,q)$ approaches $\RR^3-\W$ as a geometric limit.  

A regular ideal octahedron is obtained by gluing two square pyramids, which we will call the {\em top} and {\em bottom} pyramids. The manifold $\RR^3-\W$ is cut into square pyramids, which are glued into ideal octahedra, by a decomposition similar to that in Lemma~\ref{lemma:polyhedra}.
We give a sketch of the decomposition here; a more detailed description is given in \cite{ckp:gmax}. 

First, note that $\RR^3-\W$ admits a $\ZZ^2$ symmetry. The quotient of $\RR^3-\W$ by $\ZZ^2$ gives a link in the manifold $T^2\times [-1,1]$, with four strands forming a square on the torus $T^2$, with alternating crossings. As in the proof of Lemma~\ref{lemma:polyhedra}, take an edge for each of these four crossings. For each of the strands, take an additional edge running from that strand to $T^2\times\{+1\}$. These crossing edges bound squares on the projection plane $T^2\times\{0\}$. The additional edges give boundary edges of a \emph{top} pyramid. Similarly, edges running from strands on $T^2\times\{0\}$ to $T^2\times\{-1\}$ give edges of \emph{bottom} pyramids. When we apply the $\ZZ^2$ action, we obtain a division of $\RR^3-\W$ into $\tilde{X_1}$, obtained by gluing top pyramids along triangular faces, and $\tilde{X_2}$, obtained by gluing bottom pyramids. A fundamental domain $\P_{\W}$ for $R^3-\W$ in $\HH^3$ is explicitly obtained by attaching each top pyramid of $\tilde{X_1}$ to a bottom pyramid of $\tilde{X_2}$ along their common square face, obtaining an octahedron. In \cite[Theorem~3.1]{ckp:gmax}, we show that a complete hyperbolic structure on $\RR^2-\W$ is obtained when each octahedron is given the structure of a regular ideal octahedron. Thus the universal cover of $\RR^3-\W$ is obtained by tesselating $\HH^3$ with ideal octahedra. Figure~\ref{fig:W_top}(a) shows how the square pyramids in $\tilde{X_1}$ are viewed from infinity on the $xy$-plane.

An appropriate $\pi/2$ rotation is needed when gluing the square faces of $\tilde{X_1}$ and $\tilde{X_2}$, which determines how adjacent triangular faces are glued to obtain $\P_{\W}$.  Figure~\ref{fig:weave-decom} shows the face pairings for the triangular faces of the bottom square pyramids, and the associated circle pattern. The face pairings are equivariant under the translations $(x,y) \mapsto (x\pm 1,y\pm 1)$. That is, when a pair of faces is identified, then the corresponding pair of faces under this translation is also identified.

The proof below provides the geometric limit of the polyhedra described in Section~\ref{sec:triangulation}.  We will see that these polyhedra converge as follows.  If we cut the torus in Figure~\ref{fig:polyhedra} in half along the horizontal plane shown, each half is tessellated mostly by square pyramids, as well as some tetrahedra.  As $p,q \to \infty$, the tetrahedra are pushed off to infinity, and the square pyramids converge to the square pyramids that are shown in Figure~\ref{fig:W_top}. Gluing the two halves of the torus along the square faces of the square pyramids, in the limit we obtain the tessellation by regular ideal octahedra.

\begin{figure}[h]
\begin{tabular}{ccc}
 \includegraphics[scale=0.5]{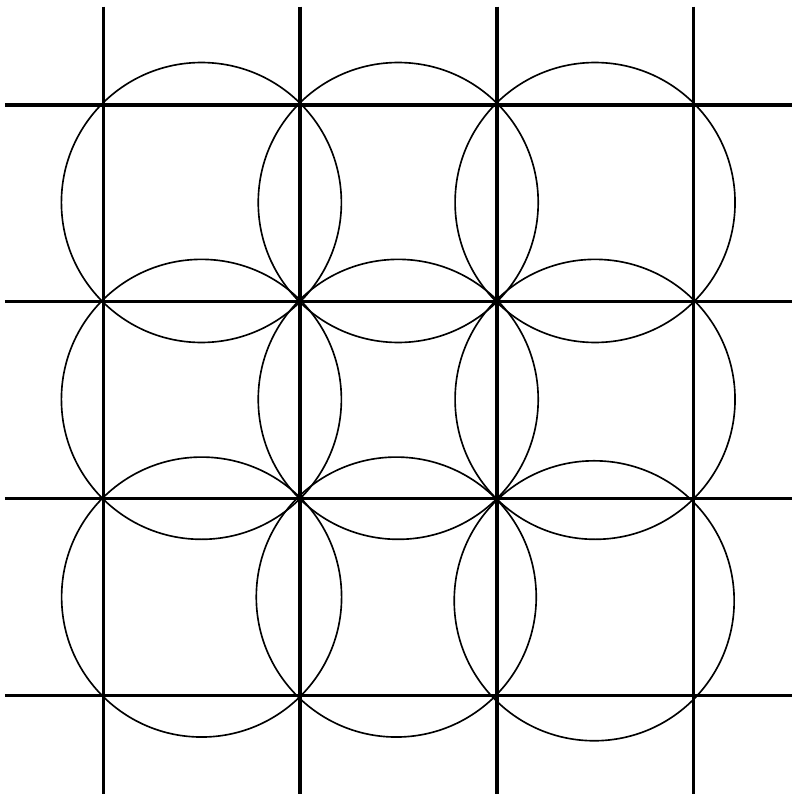} & \qquad &
 \includegraphics[scale=1.1]{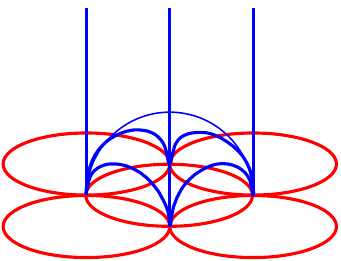} \\
(a) & \qquad & (b) \\
\end{tabular}
\caption{(a) Circle pattern for hyperbolic planes of the top polyhedron of $\R^3-\W$.
  (b) Hyperbolic planes bounding one top square pyramid.}
  \label{fig:W_top}
\end{figure}

\begin{figure}
  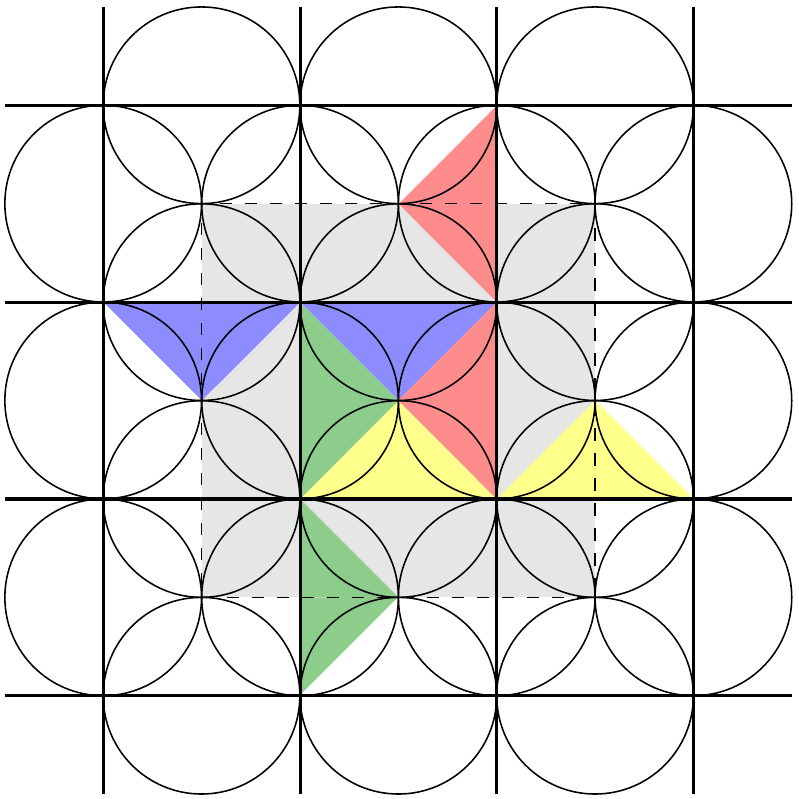\\
  \caption{Face pairings for a fundamental domain $\P_{\W}$ of $\R^3-\W$. 
}
  \label{fig:weave-decom}
\end{figure}

To make this precise, we review the definition of a geometric limit. We use bilipschitz convergence (also called ``quasi-isometry'' in \cite{BenedettiPetronio}). Convergence of metric spaces was studied in detail by Gromov \cite{Gromov}. Careful treatments of geometric limits in the hyperbolic case are given, for example, in \cite{CanaryEpsteinGreen:Notes} and \cite[Chapter~E]{BenedettiPetronio}. The formulation below will suffice for our purposes. 

\begin{define}\label{def:bilip}
For compact metric spaces $X$ and $Y$, define their \emph{bilipschitz distance} to be
\[
\inf \{ |\log \lip(f) | + |\log \lip(f^{-1})| \}
\]
where the infimum is taken over all bilipschitz mappings $f$ from $X$ to $Y$, and $\lip(f)$ denotes the lipschitz constant;
i.e.,\ the minimum value of $K$ such that
\[\frac{1}{K} \, d(x, y) \leq d(f(x),f(y) \leq K\,d(x, y)\]
for all $x, y$ in $X$. 
The lipschitz constant is defined to be infinite if there is no bilipschitz map between $X$ and $Y$.

A sequence $\{(X_n,x_n)\}$ of locally compact complete length metric spaces with distinguished basepoints is said to \emph{converge in the pointed bilipschitz topology} to $(Y,y)$ if for any $R>0$, the closed balls $B_R(x_n)$ of radius $R$ about $x_n$ in $X_n$ converge to the closed ball $B_R(y)$ about $y$ in $Y$ in the bilipschitz topology; i.e.,\ in the topology on the space of all compact metric spaces given by bilipschitz distance.
\end{define}

\begin{define}\label{def:geomlimit}
For $X_n$, $Y$ locally compact complete metric spaces, we say that $Y$ is a \emph{geometric limit} of $X_n$ if there exist basepoints $y\in Y$ and $x_n \in X_n$ such that $(X_n, x_n)$ converges in the pointed bilipschitz topology to $(Y,y)$.
\end{define}

In order to prove Theorem~\ref{thm:geolimit}, we will consider
$M_{p,q} := S^3-(W(p,q)\cup B)$.
First we will show $\RR^3-\W$ is a geometric limit of $M_{p,q}$, and then use this to show that it follows that $\RR^3-\W$ is a geometric limit of $S^3-W(p,q)$. To show $\RR^3-\W$ is a geometric limit of $M_{p,q}$, we need to find basepoints $x_{p,q}$ for each $M_{p,q}$ so that closed balls $B_R(x_{p,q})$ converge to a closed ball in $\RR^3-\W$. We do this by considering structures on ideal polyhedra.

Let $\P_{p,q}$ denote the collection of ideal polyhedra in the decomposition of $M_{p,q}$ from the proofs of Lemma~\ref{lemma:polyhedra} and Lemma~\ref{lemma:angle-struct}.
The polyhedra of $\P_{p.q}$ consist of ideal tetrahedra and ideal octahedra, such that octahedra corresponding to non-peripheral squares of the diagram projection graph of $W(p,q)$ satisfy the same local gluing condition on the faces as that for the octahedra for $\R^3-\W$ as illustrated in Figure~\ref{fig:weave-decom}.
In particular, the faces of each octahedron are glued to faces of adjacent octahedra, with the gluings of the triangular faces of the top and bottom square pyramids locally the same as those for $\P_{\W}$.

We find a sequence of consecutive octahedra in $M_{p,1} = S^3-(W(p,1)\cup B)$ with volume approaching $\voct$, and then use the $q$--fold cover $M_{p,q} \to M_{p,1}$ to find a grid of octahedra in $M_{p,q}$ all of which have volume nearly $\voct$.

\begin{lemma} \label{lem:big-oct-row}
There exist $k \to \infty,\ \epsilon(k) \to 0$, and $n(k) \to \infty$ such that for $p\geq n(k)$ there exist at least $k$ consecutive ideal octahedra in $\P_{p,1}$ 
each of which has volume greater than $(\voct - \epsilon(k))$.
\end{lemma}

\begin{proof}
Let $\epsilon(k) = \frac{1}{k}$ and $n(k) = k^3$. Suppose there are no $k$ consecutive octahedra each of whose volume is greater than $\voct-\epsilon(k)$. This implies that there exist at least $n(k)/k = k^2$ octahedra each of which has volume at most $\voct-\epsilon(k)$. Hence for $k > 12$ and $p > n(k)$,
\begin{align*}
\vol(M_{p,1}) &  \leq   4\vtet +(p-k^2)\voct + k^2(\voct -1/k) \\
&=   4\vtet + p \voct - k \\
&=  (p-2)\voct + 4\vtet + 2\voct -k \\ 
& <   (p-2) \voct .
\end{align*}
This contradicts Theorem \ref{thm:WUaxis}, which says that $(p-2) \voct < \vol(M_{p,1})$.
\end{proof}

\begin{corollary}\label{cor:big-oct-grid}
For any $\epsilon>0$ and any $k>0$ there exists $N$ such that if $p,q>N$ then $\P_{p,q}$ contains a $k\times k$ grid of adjacent ideal octahedra, each of which has volume greater than $(\voct - \epsilon)$.
\end{corollary}

\begin{proof}
Apply Lemma~\ref{lem:big-oct-row}, taking $k$ sufficiently large so that $\epsilon(k)<\epsilon$.  Then for any $N>n(k)$, if $p>N$ we obtain at least $k$ consecutive ideal octahedra with volume as desired. Now let $q>N$, so $q>k$.  Use the $q$--fold cover $M_{p,q} \to M_{p,1}$. We obtain a $k\times q$ grid of octahedra, all of which have volume greater than $(\voct-\epsilon(k))$, as shown in Figure~\ref{fig:big-oct-grid}.
\end{proof}

\begin{figure}[h]
\includegraphics[height=1in]{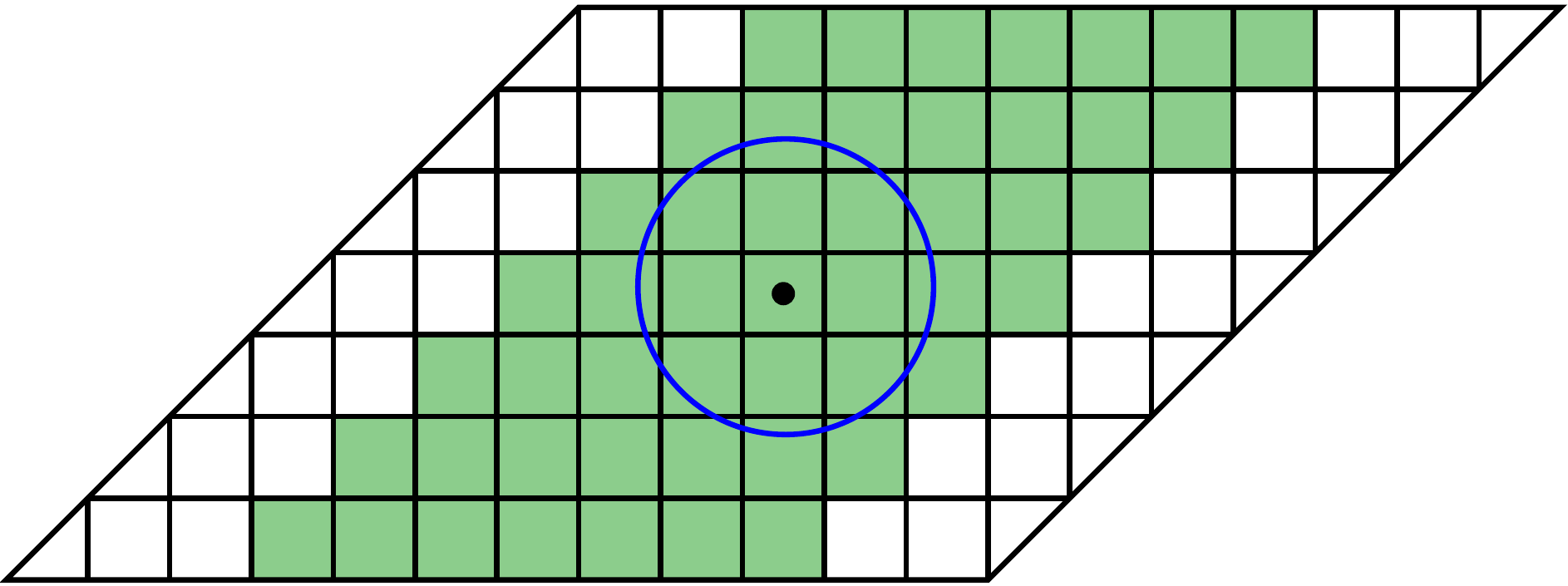}
\caption{Grid of octahedra with volumes near $\voct$ in $\P_{p,q}$, and 
base point. }
\label{fig:big-oct-grid}
\end{figure}

We are now ready to complete the proof of Theorem~\ref{thm:geolimit}.

\begin{proof}[Proof of Theorem~\ref{thm:geolimit}]
Given $R>0$, we will show that closed balls based in $M_{p,q}$ converge to a closed ball based in $\RR^3-\W$.

Take a basepoint $y \in \RR^3-\W$ to lie in the interior of any octahedron, on one of the squares projecting to a checkerboard surface, say at the point where the diagonals of that square intersect. Consider the ball $B_R(y)$ of radius $R$ about the basepoint $y$. This will intersect some number of regular ideal octahedra. Notice that the octahedra are glued on all faces to adjacent octahedra, by the gluing pattern we obtained in Figure~\ref{fig:weave-decom}.  Consider all octahedra in $\RR^3-\W$ that meet the ball $B_R(y)$.
Call this collection of octahedra $\Oct(R)$.

In $M_{p,q}$, consider an octahedron of Lemma~\ref{lemma:polyhedra} coming from a square in the interior of the diagram of $W(p,q)$, so that the octahedron is glued only to other octahedra in the polyhedral decomposition.
Then the gluing pattern on each of its faces agrees with the gluing of octahedra in $\RR^3-\W$. Thus for $p$, $q$ large enough, we may find a collection of adjacent octahedra $\Oct_{p,q}$ in $M_{p,q}$ with the same combinatorial gluing pattern as $\Oct(R)$. Since all the octahedra are glued along faces to adjacent octahedra, Corollary~\ref{cor:big-oct-grid} implies that if we choose $p$, $q$ large enough, then each ideal octahedron in $\Oct_{p,q}$ has volume within $\epsilon$ of $\voct$. 

It is known that the volume of a hyperbolic ideal octahedron is uniquely maximized by the volume of a regular ideal octahedron (see, e.g.\ \cite[Theorem~10.4.7]{ratcliffe}). Thus as $\epsilon\to 0$, each ideal octahedron of $\Oct_{p,q}$ must be converging to a regular ideal octahedron. So $\Oct_{p,q}$ converges as a polyhedron to $\Oct(R)$. But then it follows that for suitable basepoints $x_{p,q}$ in $\P_{p,q}$, the balls $B_R(x_{p,q})$ in $\P_{p,q} \subset M_{p,q}$ converge to $B_R(y)$ in the pointed bilipschitz topology.

Finally, we use the fact that $M_{p,q} = S^3-(W(p,q)\cup B)$ converges to $\RR^3-\W$ geometrically to show that $S^3-W(p,q)$ also converges to $\RR^3-\W$ geometrically. This will follow from the drilling/filling theorems of Hodgson and Kerckhoff \cite{HodgsonKerckhoff}, Brock and Bromberg \cite{BrockBromberg}, using the formulation of Magid \cite{Magid}. Recall that $S^3-W(p,q)$ is obtained from $M_{p,q}$ by Dehn filling the meridian slope of the cusp of $M_{p,q}$ corresponding to the braid axis $B$ of $W(p,q)$. The drilling and filling theorems control geometry change under Dehn filling, provided the normalized length of the Dehn filling slope is sufficiently long, where \emph{normalized length} is defined to be the length of the slope divided by the square root of the area of the cusp torus containing the slope. 

By Lemma~\ref{lemma:meridian-length}, the length of the Dehn filling slope is at least $q$. Moreover, because the braid axis of $W(p,q)$ is $q$-fold covered by the braid axis of $W(p,1)$, the area of the cusp corresponding to $B$ of $M_{p,q}$ is $q$ times the area of the cusp corresponding to the braid axis for $M_{p,1}$. Thus the normalized length of the slope is at least $c\,\sqrt{q}$, where $c$ is some positive constant. It follows that as $q$ goes to infinity, normalized length also goes to infinity.

Now apply \cite[Theorem~1.2]{Magid}. This theorem states that for any
bilipschitz constant $J>1$, and any $\epsilon>0$, there is a universal
constant $K$ such that for any geometrically finite hyperbolic
3-manifold $\hat{M}$ with no annular cusp, and with a distinguished
torus cusp $T$ (in our case, $\hat{M}=M_{p,q}$ with cusp $T=B$),
and any slope $\beta$ on $T$ with normalized length at least $K$, there
exists a $J$-bilipschitz diffeomorphism from the complement of an
$\epsilon$-thin neighborhood of the cusp $T$ in $\hat{M}$ to the
complement of an $\epsilon$-thin neighborhood of the tube about the
filled curve in the filled manifold $\hat{M}(\beta)$.  Since we
already know that compact balls in $M_{p,q}$ converge to those in
$\RR^3-\W$ in the pointed bilipschitz topology, it follows that by
choosing appropriate sequences of $\epsilon$ and $J$ in the filling
theorem, and letting $q\to\infty$, corresponding compact balls in
$S^3-W(p,q)$ also converge to those in $\RR^3-\W$, and thus $\RR^3-\W$
is a geometric limit of $S^3-W(p,q)$.
\end{proof}

\bibliographystyle{amsplain} \bibliography{references}

\end{document}